\documentclass[a4paper,14pt]{article}
\def\ds{\displaystyle}
\def\ni{\noindent}
\usepackage[a4paper, total={6in, 8in}]{geometry}
\usepackage{setspace}
\usepackage{amssymb}
\usepackage[utf8]{inputenc}
\usepackage[english]{babel}
\usepackage{amsthm}
\usepackage{amsmath}
\allowdisplaybreaks[1]

\newtheorem{theorem}{Theorem} [section]
\newtheorem{lemma}{Lemma} [section]

\newtheorem{corollary}{Corollary}

\newtheorem{proposition}{Proposition}

\numberwithin{equation}{section}

\begin{document}
\begin{center}
{\bf \Large Ramanujan type of congruences modulo $m$ for $(l, m)$-regular bipartitions }
\end{center}
\begin{center}
\bf T. Kathiravan
\footnotetext{Email: kkathiravan98@gmail.com}
\end{center}
{\baselineskip .5cm \begin{center}The Institute of Mathematical Sciences,\\HBNI,\\ CIT Campus, Taramani,\\ Chennai 600113, India.\\
\end{center}}
\ni {\bf \small Abstract:}
\par Let $B_{l,m}(n)$ denote the number of $(l,m)$-regular bipartitions of $n$. Recently, many authors proved several infinite families of congruences modulo $3$, $5$ and $11$ for $B_{l,m}(n)$. 
In this paper, using theta function identities to prove infinite families of congruences modulo $m$ for $(l,m)$-regular bipartitions, where $m\in\{7,3,11,13,17\}$.\\

\ni {\bf \small 2010 Mathematics Subject Classification:} 11P83, 05A17.\\
\ni {\bf \small Keywords:} Congruence, Regular Bipartition.

\section{Introduction}
For a positive integer $\ell$, a partition is said to be $\ell$-regular if none of its parts is divisible by $\ell$. 
If $b_\ell(n)$ denote the number of $\ell$-regular partitions of $n$, then the generating function for $b_\ell(n)$, is given by (see \cite{Andrews1})
 \begin{equation}\label{0.a}
\sum_{n=0}^{\infty}b_{\ell}(n)q^n=\frac{f_\ell}{f_1}.
\end{equation}
where as customary, we define
\begin{equation}\nonumber
f_k:=(q^k;q^k)_\infty=\prod^\infty_{m=1}(1-q^{mk}).
\end{equation}
In recent years, many authors studied arithmetic properties of $\ell$-regular partitions \cite{Carlson, Cui, Cui1, Dandurand, Furcy, Gordon, Hirschhorn, Keith, Webb, Xia3, Xia2, Xia, Yao}.\\
Recall that, for a positive integers $l>1$ and $m>1$, a bipartition $(\lambda,\mu)$ of $n$ is a pair of partitions $(\lambda,\mu)$ such that the sum of all the parts equals $n$. 
A $(l,m)$-regular bipartition of $n$ is a bipartition $(\lambda,\mu)$ of $n$ such that $\lambda$ is a $l$-regular partition and $\mu$ is a $m$-regular partition. 
If $B_{l,m}(n)$ denote the number of $(l,m)$-regular bipartitions of $n$, then the generating function $B_{l,m}(n)$, is given by (see \cite{Lin1})
\begin{equation}\label{0.1}
\sum_{n=0}^{\infty}B_{l,m}(n)q^n=\frac{f_lf_m}{f^2_1}.
\end{equation}
In \cite{Dou}, Dou proved that, for $n\geq0$ and $\alpha\geq2$,
\begin{equation*}
B_{3,11}\left(3^\alpha n+\frac{5\cdot3^{\alpha-1}-1}{2}\right)\equiv0\pmod{11}.
\end{equation*}
The reader may refer to \cite{Kathiravan, Lin1, Lin2, Wang} for related works. Recently, Adiaga and  Ranganatha \cite{Adiaga} proved infinite families of congruences modulo $3$ for $B_{3,7}(n)$. 
More recently, Xia and Yao \cite{Xia1} proved several infinite families of congruences modulo $3$ for $B_{3,s}(n)$, modulo $5$ for $B_{5,s}(n)$ and modulo $7$ for $B_{3,7}(n)$. For example,
let $s$ be a positive integer and let $p\geq5$ be a prime, for $n\geq0$,
\begin{equation*}
B_{3,s}\left(p^{2\alpha+1}n+\frac{(1+s)(p^{2\alpha+2}-1)}{24}\right)\equiv0\pmod3.
\end{equation*}
The aim of this paper is to prove several infinite families of congruences modulo $7$ for $B_{3,7}(n)$, modulo $3$ for $B_{9,5}(n)$, modulo $11$ for $B_{5,11}(n)$, modulo $13$ for $B_{5,13}(n)$ and modulo $17$ for $B_{81,17}(n)$. The main results of this paper are as follows,
\begin{theorem}\label{ak}\normalfont For all $n\geq0$ and $m\geq0$,
\begin{align}
\label{ak1}
B_{3,7}\left(4^{7m}n+\frac{4^{7m}-1}{3}\right)&\equiv3^mB_{3,7}(n)\pmod{7},\\\label{ak2}
B_{3,7}\left(4^{7m+7}n+\frac{10\cdot4^{7m+6}-1}{3}\right)&\equiv0\pmod{7}.
\end{align}
\end{theorem}
\begin{theorem}\label{0a}\normalfont For all $n\geq0$ and $m\geq0$,
\begin{align}\normalfont\label{0a1}
B_{9,5}\left(5^{4m}n+\frac{5^{4m}-1}{2}\right)&\equiv2^mB_{9,5}(n)\pmod3,\\
B_{9,5}\left(5^{4m+4}n+\frac{(2k+1)\cdot5^{4m+3}-1}{2}\right)&\equiv 0\pmod3, ~~\textrm{where} ~~k\in\{4,5\}\label{0a2}.
\end{align}
\end{theorem}
\begin{theorem}\label{1a}\normalfont For all $n\geq0$ and $m\geq0$,
\begin{eqnarray}\label{12}
B_{5,11}\left(5^{12m}n+\frac{7\cdot5^{12m}-7}{12}\right)&\equiv&2^mB_{5,11}(n)\pmod{11}.\\
\label{13}
B_{5,11}\left(5^{12m+12}+\frac{(12k+11)\cdot5^{12m+11}-7}{12}\right)&\equiv&0\pmod{11},~~\textrm{where}~~k\in\{4,5\}.
\end{eqnarray}
\end{theorem}
\begin{theorem}\label{1b}\normalfont For all $n\geq0$ and $m\geq0$,
\begin{eqnarray}\label{14}
B_{5,13}\left(5^{6m}n+\frac{2\cdot5^{6m}-2}{3}\right)&\equiv&8^mB_{5,13}(n)\pmod{13}.\\
\label{15}
B_{5,13}\left(5^{6m+6}n+\frac{(3k+1)\cdot5^{6m+5}-2}{3}\right)&\equiv&0\pmod{13},~~\textrm{where}~~k\in\{1,5\}.
\end{eqnarray}
\end{theorem}
\begin{theorem}\label{1c}\normalfont  For all $n\geq0$ and $m\geq0$,
\begin{eqnarray}
B_{81,17}\left(81\cdot4^{9m}n+81\frac{(2\cdot4^{8m}-2)}{3}+50\right)&\equiv&0,\pmod{17}\label{s13}\\
B_{81,17}\left(81\cdot4^{9m}n+81\frac{(2\cdot4^{9m}-2)}{3}+50\right)&\equiv&8^mB_{81,17}(81n+50),\pmod{17}\label{s14}\\
B_{81,17}\left(162\cdot4^{8m}n+81\frac{(5\cdot4^{8m}-2)}{3}+50\right)&\equiv&5^mB_{81,17}(162n+131).\pmod{17}\label{s15}
\end{eqnarray}
\end{theorem}
\begin{theorem}\label{1d}\normalfont For all $n\geq0$,
\begin{equation}\label{x1}
B_{2,8}(8(11n+k)+7)\equiv0\pmod{11},~~\textrm{where}~~1\leq k\leq10.
\end{equation}
\end{theorem}
The paper is organized as follows. In section $2$, we state some preliminary results. The proof of the theorems are carried out in sections $3$ - $8$.
\section{Preliminaries}
In this section, we collect some dissection formulas which are useful to prove our main results.\\
For $\mid ab\mid<1$, Ramanujan's general theta function $f(a,b)$ is defined as
\begin{equation} \label{a}
f(a,b)=\ds\sum^\infty_{n=-\infty}a^{n(n+1)/2}b^{n(n-1)/2}.
\end{equation}
Using Jacobi's triple product identity \cite[Entry 19, p. 35]{Berndt} in \eqref{a},
\begin{equation}
f(a,b)=(-a;ab)_\infty(-b;ab)_\infty(ab;ab)_\infty.
\end{equation}
The most important three special cases of $f(a,b)$ are
\begin{equation}
\varphi(q):=f(q,q)=1+2\ds\sum^\infty_{n=1}q^{n^2}=(-q;q^2)^2_\infty(q^2;q^2)_\infty=\frac{f^5_2}{f^2_1f^2_4},
\end{equation}
\begin{equation}
\psi(q):=f(q;q^3)=\ds\sum^\infty_{n=0}q^{n(n+1)/2}=\frac{(q^2;q^2)_\infty}{(q;q^2)_\infty}=\frac{f^2_2}{f_1},
\end{equation}
and
\begin{equation} \label{b}
f(-q):=f(-q,-q^2)=\ds\sum^\infty_{n=-\infty}(-1)^nq^{n(3n-1)/2}=(q;q)_\infty=f_1.
\end{equation}
where the product representations \eqref{b} is Euler's famous pentagonal theorem \cite{Andrews}. \\
By the binomial theorem for any prime $p$,
\begin{equation}\label{k1}
f_p\equiv f^p_1\pmod p.
\end{equation}
\begin{lemma}\normalfont\cite[Eqs. (8.4.1), (8.4.2) and (8.4.4)]{H1} If $S:=R(q^5):=\ds\frac{(q^5;q^{25})_\infty(q^{20};q^{25})_\infty}{(q^{10};q^{25})_\infty(q^{15};q^{25})_\infty}$,
than
\begin{equation}\label{0.2}
f_1=f_{25}\ds\left(\ds\frac{1}{S}-q-q^2S\right),
\end{equation}
\begin{equation}\label{0.3}
\ds\frac{f^6_5}{f^6_{25}}=\ds\left(\ds\frac{1}{S^5}-11q^5-q^{10}S^5\right)
\end{equation}
and
\begin{equation}\label{0.3a}
\ds\frac{1}{f_1}=\ds\frac{f^5_{25}}{f_5^6}\left(q^8 S^4-q^7 S^3+2 q^6 S^2-3 q^5 S+5 q^4+\ds\frac{3 q^3}{S}
+\ds\frac{2 q^2}{S^2}+\ds\frac{q}{S^3}+\ds\frac{1}{S^4}\right)
\end{equation}
\end{lemma}
\begin{lemma}\normalfont In \cite[p.40]{Berndt}, we have
\begin{eqnarray}\label{a}
  \ds\frac{1}{f^2_1}&=&\frac{f^5_8}{f^5_2f^2_{16}}+2q\frac{f^2_4f^2_{16}}{f_2^5f_8},\\\label{b}
  \ds\frac{1}{f^4_1}&=&\frac{f_4^{14}}{f^{14}_2f_8^4}+4q\frac{f_4^2f_8^4}{f^{10}_2}\\\label{e2}
  \ds f_1^4&=&\frac{f_4^{10}}{f_2^2 f_8^4}-4q\frac{f_2^2 f_8^4}{f_4^2}.
\end{eqnarray}
\end{lemma}
\begin{lemma}\normalfont In \cite[Lemma 2]{Naika}, we have
\begin{eqnarray}\label{kp}
  \ds\frac{f_3}{f_1^3}&=&\frac{f_4^6f_6^3}{f_2^9 f_{12}^2}+3q\frac{f_4^2f_6 f_{12}^2}{f_2^7}\\\label{kp1}
  \ds\frac{f_1^3}{f_3}&=&\frac{f_4^3}{f_{12}}-3q\frac{f_2^2 f_{12}^3}{f_4 f_6^2}
\end{eqnarray}
\end{lemma}
\begin{lemma}\normalfont In \cite[Lemma 3]{Naika}, we have
\begin{eqnarray}\label{kp2}
\ds f_1 f_3&=&\frac{f_2 f_8^2 f_{12}^4}{f_4^2 f_6 f_{24}^2}-q\frac{f_4^4 f_6 f_{24}^2}{f_2 f_8^2 f_{12}^2}
\end{eqnarray}
\end{lemma}
\begin{lemma}\normalfont In \cite{Xia4}, we have
\begin{eqnarray}\label{kp3}
\frac{f_3^2}{f_1^2}&=&\frac{f_6 f_{12}^2 f_4^4}{f_2^5 f_8 f_{24}}+2q\frac{f_6^2 f_8 f_{24} f_4}{f_2^4 f_{12}}
\end{eqnarray}
\end{lemma}
\section{Congruence for $(3,7)$-regular bipartition}
In order to prove Theorem \ref{ak}, we need some basic results which we shell develop below. The proof, will then follow as a consequence of these argnments. 
\begin{lemma}\label{kpz}\normalfont For all $n\geq0$ and $m\geq0$, we have
\begin{equation}\label{kp0}
B_{3,7}\left(4^{m}n+\frac{4^{m}-1}{3}\right)\equiv E_mB_{3,7}(4n+1)+e_mB_{3,7}(n)
\end{equation}
where
\begin{equation}\nonumber
\ds E_m=\frac{\sqrt{14}}{28}\left(3+\sqrt{14}\right)^m-\frac{\sqrt{14}}{28}(3-\sqrt{14})^m
\end{equation}
and
\begin{equation}\nonumber
\ds e_m=\left(\frac{1}{2}-\frac{3}{2\sqrt{14}}\right)(3+\sqrt{14})^m+\left(\frac{1}{2}+\frac{3}{2\sqrt{14}}\right)(3-\sqrt{14})^m.
\end{equation}
\end{lemma}
\begin{proof}
Setting $l=3, m=7$ in \eqref{0.1}, we have
\begin{equation}\label{w.1}
\ds\sum_{n=0}^{\infty}B_{3,7}(n)q^n=\frac{f_3f_7}{f_1^2}.
\end{equation}
From \eqref{k1} and \eqref{w.1}, we have
\begin{equation}\label{w.2}
\ds\sum_{n=0}^{\infty}B_{3,7}(n)q^n\equiv f^5_1f_3 \pmod7.
\end{equation}
Substituting \eqref{e2} and \eqref{kp2} into \eqref{w.2}, we have
\begin{align}\label{w.3}\nonumber
\ds\sum_{n=0}^{\infty}B_{3,7}(n)q^n&\equiv\left(\frac{f_4^{10}}{f_2^2 f_8^4}-4q\frac{f_2^2 f_8^4}{f_4^2}\right)\left(\frac{f_2 f_8^2 f_{12}^4}{f_4^2 f_6 f_{24}^2}-q\frac{f_4^4 f_6 f_{24}^2}{f_2 f_8^2 f_{12}^2}\right)\\
&\equiv\frac{f_4^8 f_{12}^4}{f_2 f_6 f_8^2 f_{24}^2}+3q\frac{f_2^3 f_8^6 f_{12}^4}{f_4^4 f_6 f_{24}^2}+6q\frac{f_4^{14} f_6 f_{24}^2}{f_2^3 f_8^6 f_{12}^2}+4q^2\frac{f_2 f_4^2 f_6 f_8^2 f_{24}^2}{f_{12}^2}
\end{align}
If we extract those terms in which the power of $q$ is congruent to 1 modulo 2,  divide by $q$ and replace $q^2$ by $q$, we have
\begin{align}\label{w.4}
\ds\sum_{n=0}^{\infty}B_{3,7}(2n+1)q^n&\equiv3\frac{f_1^3 f_4^6 f_{6}^4}{f_2^4 f_3f_{12}^2}+6\frac{f_2^{14}f_3f_{12}^2}{f_1^3f_4^6f_{6}^2}
\end{align}
Substituting \eqref{kp} and \eqref{kp1} into \eqref{w.4}, we have
\begin{align}\label{w.5}\nonumber
\ds\sum_{n=0}^{\infty}B_{3,7}(2n+1)q^n&\equiv3\frac{f_4^6 f_{6}^4}{f_2^4f_{12}^2}\left(\frac{f_4^3}{f_{12}}-3q\frac{f_2^2 f_{12}^3}{f_4 f_6^2}\right)+6\frac{f_2^{14}f_{12}^2}{f_4^6f_{6}^2}\left(\frac{f_4^6f_6^3}{f_2^9 f_{12}^2}+3q\frac{f_4^2f_6 f_{12}^2}{f_2^7}\right),\\
&\equiv6 f_2^5 f_6+3\frac{f_4^9 f_6^4}{f_2^4 f_{12}^3}+5q\frac{f_4^5 f_6^2 f_{12}}{f_2^2}+4q\frac{f_2^7 f_{12}^4}{f_4^4 f_6}.
\end{align}
If we extract those terms in which the power of $q$ is congruent to 0 modulo 2, replace $q^2$ by $q$, we have
\begin{align}\label{w.6}
\ds\sum_{n=0}^{\infty}B_{3,7}(4n+1)q^n&\equiv6 f_1^5 f_3+3\frac{f_2^9 f_3^4}{f_1^4 f_{6}^3}.
\end{align}
Substituting \eqref{e2}, \eqref{kp2} and \eqref{kp3} into \eqref{w.6}, we have
\begin{align}\label{w.7}\nonumber
\ds\sum_{n=0}^{\infty}B_{3,7}(4n+1)q^n&\equiv6\left(\frac{f_4^{10}}{f_2^2 f_8^4}-4q\frac{f_2^2 f_8^4}{f_4^2}\right)\left(\frac{f_2 f_8^2 f_{12}^4}{f_4^2 f_6 f_{24}^2}-q\frac{f_4^4 f_6 f_{24}^2}{f_2 f_8^2 f_{12}^2}\right)+3\frac{f_2^9}{f_{6}^3}\left(\frac{f_6 f_{12}^2 f_4^4}{f_2^5 f_8 f_{24}}+2q\frac{ f_6^2 f_8 f_{24} f_4}{f_2^4 f_{12}}\right)^2\\
&\equiv2\frac{f_4^8 f_{12}^4}{f_2 f_6 f_8^2 f_{24}^2}+5q f_4^5 f_{12}+4q\frac{f_2^3 f_8^6 f_{12}^4}{f_4^4 f_6 f_{24}^2}+q\frac{f_4^{14} f_6 f_{24}^2}{f_2^3 f_8^6 f_{12}^2}+q^2\frac{f_2 f_4^2 f_6 f_8^2 f_{24}^2}{f_{12}^2}
\end{align}
If we extract those terms in which the power of $q$ is congruent to 1 modulo 2,  divide by $q$ and replace $q^2$ by $q$, we have
\begin{align}\label{w.8}
\ds\sum_{n=0}^{\infty}B_{3,7}(8n+5)q^n&\equiv5f_2^5 f_{6}+4\frac{f_1^3 f_4^6 f_{6}^4}{f_2^4 f_3f_{12}^2}+\frac{f_2^{14}f_3f_{12}^2}{f_1^3 f_4^6 f_{6}^2}.
\end{align}
Substituting \eqref{kp} and \eqref{kp1} into \eqref{w.7}, we have
\begin{align}\label{w.9}\nonumber
\ds\sum_{n=0}^{\infty}B_{3,7}(8n+5)q^n&\equiv5f_2^5 f_{6}+4\frac{f_4^6 f_{6}^4}{f_2^4f_{12}^2}\left(\frac{f_4^3}{f_{12}}-3q\frac{f_2^2 f_{12}^3}{f_4 f_6^2}\right)+\frac{f_2^{14}f_{12}^2}{f_4^6f_{6}^2}\left(\frac{f_4^6f_6^3}{f_2^9 f_{12}^2}+3q\frac{f_4^2f_6 f_{12}^2}{f_2^7}\right),\\
&\equiv6f_2^5 f_6+4\frac{f_4^9 f_6^4}{f_2^4 f_{12}^3}+2q\frac{f_4^5 f_6^2 f_{12}}{f_2^2}+3q\frac{f_2^7 f_{12}^4}{f_4^4 f_6}.
\end{align}
If we extract those terms in which the power of $q$ is congruent to 0 modulo 2, replace $q^2$ by $q$, we have
\begin{align}\label{w.10}
\ds\sum_{n=0}^{\infty}B_{3,7}(16n+5)q^n&\equiv6f_1^5 f_3+4\frac{f_2^9 f_3^4}{f_1^4 f_{6}^3}.
\end{align}
From \eqref{w.2}, \eqref{w.6} and \eqref{w.10}, we have
\begin{align}\label{w.11}
B_{3,7}(16n+5)&\equiv5B_{3,7}(n)+6B_{3,7}(4n+1).
\end{align}
Which is true if $m=1$. Now, we prove by induction on $k$ in \eqref{kp0}. Assume that Lemma \ref{kpz} is true for $k=m-1$ and $k=m$,
\begin{equation}\label{w.12}
B_{3,7}\left(4^{m-1}n+\frac{4^{m-1}-1}{3}\right)\equiv E_{m-1}B_{3,7}(4n+1)+e_{m-1}B_{3,7}(n),
\end{equation}
and
\begin{equation}\label{w.13}
B_{3,7}\left(4^{m}n+\frac{4^{m}-1}{3}\right)\equiv E_mB_{3,7}(4n+1)+e_mB_{3,7}(n).
\end{equation}
Now, we will show that $k=m+1$ is true. Replacing $n=\ds4^{m-1}+\frac{4^{m-1}-1}{3}$ in \eqref{w.11}, we have
\begin{align}\label{w.14}\nonumber
B_{3,7}\left(4^{m+1}n+\frac{4^{m+1}-1}{3}\right)&\equiv6B_{3,7}\left(4^{m}n+\frac{4^{m}-1}{3}\right)\\
&\qquad+5B_{3,7}\left(4^{m-1}n+\frac{4^{m-1}-1}{3}\right).
\end{align}
we can easily check that
\begin{equation}\label{w.15}
E_{m+1}=6E_m+5E_{m-1}
\end{equation}
and
\begin{equation}\label{w.16}
e_{m+1}=6e_m+5e_{m-1}
\end{equation}
From equation \eqref{w.12}, \eqref{w.13}, \eqref{w.14}, \eqref{w.15} and \eqref{w.16}, we have
\begin{align}\label{w.17}\nonumber
B_{3,7}\left(4^{m+1}n+\frac{4^{m+1}-1}{3}\right)&\equiv6(E_mB_{3,7}(4n+1)+e_mB_{3,7}(n))\\\nonumber
&\qquad+5(E_{m-1}B_{3,7}(4n+1)+e_{m-1}B_{3,7}(n)),\\\nonumber
&\equiv(6E_m+5E_{m-1})B_{3,7}(4n+1)+(6e_m+5e_{m-1})B_{3,7}(n),\\
&\equiv E_{m+1}B_{3,7}(4n+1)+e_{m+1}B_{3,7}(n).
\end{align}
Therefore the Lemma \ref{kpz} is true by induction of \eqref{kp0}.
\end{proof}
\begin{proposition}\label{kpz1}\normalfont For all $n\geq0$,
\begin{eqnarray}
\label{w.18}
B_{3,7}\left(4^7n+\frac{10\cdot4^6-1}{3}\right)&\equiv&0\pmod{7},\\
\label{w.19}
B_{3,7}\left(4^7n+\frac{4^7-1}{3}\right)&\equiv&3B_{3,7}(n)\pmod{7}.
\end{eqnarray}
\end{proposition}
\begin{proof}
From Lemma \ref{kpz}, put $m=6$, we find that, $E_6\equiv2\pmod{7}$ and $e_{6}\equiv2\pmod{7}$,
\begin{equation}\label{w.20}
\ds\sum_{n=0}^{\infty}B_{3,7}\left(4^{6}n+\frac{4^{6}-1}{3}\right)q^n\equiv6\frac{f_2^9 f_3^4}{f_1^4 f_{6}^3}.
\end{equation}
Substituting \eqref{kp3}  into \eqref{w.20}, we have
\begin{align}\label{w.21}\nonumber
\ds\sum_{n=0}^{\infty}B_{3,7}\left(4^{6}n+\frac{4^{6}-1}{3}\right)q^n&\equiv6\frac{f_2^9}{f_{6}^3}\left(\frac{f_6 f_{12}^2 f_4^4}{f_2^5 f_8 f_{24}}+2q\frac{f_6^2 f_8 f_{24} f_4}{f_2^4 f_{12}}\right)^2,\\
&\equiv6\frac{f_4^8 f_{12}^4}{f_2 f_6 f_8^2 f_{24}^2}+3q f_4^5 f_{12}+3q^2\frac{f_2 f_4^2 f_6 f_8^2 f_{24}^2}{f_{12}^2}.
\end{align}
If we extract those terms in which the power of $q$ is congruent to 1 modulo 2,  divide by $q$ and replace $q^2$ by $q$, we have
\begin{align}\label{w.22}
\ds\sum_{n=0}^{\infty}B_{3,7}\left(2\cdot4^{6}n+\frac{4^{7}-1}{3}\right)q^n&\equiv3f_2^5 f_{6}
\end{align}
This  completes  the  proof of equation \eqref{w.18} and \eqref{w.19} follow from \eqref{w.22}.
\end{proof}
\ni\textbf{Proof of Theorem \ref{ak}.} Equation \eqref{ak1} follow from \eqref{w.19} by mathematical induction. Empolying \eqref{w.18} in \eqref{ak1}, we obtain \eqref{ak2}.
\section{Congruence for $(9,5)$-regular bipartition}
In this section we will prove Theorem \ref{0a}, we need the result.
\begin{proposition}\normalfont For $n\geq0$, we have
\begin{align}\normalfont\label{s2}
B_{9,5}\left(5^4n+\frac{5^4-1}{2}\right)&\equiv 2B_{9,5}(n)\pmod3,\\
B_{9,5}\left(5^4n+\frac{(2k+1)\cdot5^3-1}{2}\right)&\equiv 0\pmod3, ~~\textrm{where} ~~k\in\{4,5\}\label{s3}.
\end{align}
\end{proposition}
\begin{proof}
Setting $l=9, m=5$ in \eqref{0.1}, we have
\begin{equation}\label{3.1}
\ds\sum_{n=0}^{\infty}B_{9,5}(n)q^n=\frac{f_9f_5}{f_1^2}.
\end{equation}
From \eqref{k1} and \eqref{3.1}, we have
\begin{equation}\label{3.2}
\ds\sum_{n=0}^{\infty}B_{9,5}(n)q^n\equiv f^7_1f_5 \pmod3.
\end{equation}
Substituting \eqref{0.2} into \eqref{3.2}, we have
\begin{align}\label{3.3}\nonumber
\ds\sum_{n=0}^{\infty}B_{9,5}(n)q^n&\equiv f_5f^7_{25}\left(\ds\frac{1}{S}-q-q^2S\right)^7,\\\nonumber
&\equiv f_5f^7_{25}\left(\frac{1}{S^7}+\frac{2 q}{S^6}+\frac{2 q^2}{S^5}+\frac{q^3}{S^4}+\frac{2 q^4}{S^3}+\frac{2 q^5}{S^2}+\frac{2 q^6}{S}+q^7+S q^8+2S^2q^9
\right.\\
&\qquad\left.\vphantom{\ds\frac{8q^3}{S}}+ S^3q^{10}+S^4q^{11} +S^5q^{12}+2 S^6q^{13}+2 S^7q^{14}\right).
\end{align}
If we extract those terms in which the power of $q$ is congruent to 2 modulo 5,  divide by $q^2$, we have
\begin{equation}\label{3.4}
\ds\sum_{n=0}^{\infty}B_{9,5}(5n+2)q^{5n}\equiv f_5f^7_{25}\left(2\left(\frac{1}{S^5}-11 q^5-q^{10}S^5\right)+2 q^5\right).
\end{equation}
Substituting \eqref{0.3} into \eqref{3.4}, we have
\begin{align}\label{3.5}
\ds\sum_{n=0}^{\infty}B_{9,5}(5n+2)q^{5n}&\equiv f_5f^7_{25}\left(2\frac{f^6_5}{f^6_{25}}+2 q^5\right)\equiv 2f^7_5f_{25}+2f_5f^7_{25}q^5.
\end{align}
The above equation $q^5$ replace by $q$, we have
\begin{equation}\label{3.6}
\ds\sum_{n=0}^{\infty}B_{9,5}(5n+2)q^{n}\equiv 2f^7_1f_{5}+2f_1f^7_{5}q.
\end{equation}
Substituting \eqref{0.2} into \eqref{3.6}, we have
\begin{align}\label{3.7}\nonumber
\ds\sum_{n=0}^{\infty}B_{9,5}(5n+2)q^n&\equiv 2f_5f^7_{25}\left(\frac{1}{S^7}+\frac{2 q}{S^6}+\frac{2 q^2}{S^5}+\frac{q^3}{S^4}+\frac{2 q^4}{S^3}+\frac{2 q^5}{S^2}+\frac{2 q^6}{S}+q^7+S q^8+2S^2q^9+ S^3q^{10}
\right.\\
&\qquad\left.\vphantom{\ds\frac{8q^3}{S}}
+S^4q^{11} +S^5q^{12}+2 S^6q^{13}+2 S^7q^{14}\right)+2f^7_{5}f_{25}\left(\ds\frac{1}{S}-q-q^2S\right)q.
\end{align}
If we extract those terms in which the power of $q$ is congruent to 2 modulo 5,  divide by $q^2$, we have
\begin{equation}\label{3.8}
\ds\sum_{n=0}^{\infty}B_{9,5}(25n+12)q^{5n}\equiv 2f_5f^7_{25}\left(2\left(\frac{1}{S^5}-11 q^5-q^{10}S^5\right)+2 q^5\right)+f^7_{5}f_{25}.
\end{equation}
Substituting \eqref{0.3} into \eqref{3.8}, we have
\begin{align}\label{3.9}\nonumber
\ds\sum_{n=0}^{\infty}B_{9,5}(25n+12)q^{5n}&\equiv 2f_5f^7_{25}\left(2\frac{f^6_5}{f^6_{25}}+2 q^5\right)+f^7_{5}f_{25}\\
&\equiv 2f^7_5f_{25}+f_5f^7_{25}q^5.
\end{align}
The above equation $q^5$ replace by $q$, we have
\begin{equation}\label{3.10}
\ds\sum_{n=0}^{\infty}B_{9,5}(25n+12)q^{n}\equiv 2f^7_1f_{5}+f_1f^7_{5}q.
\end{equation}
Substituting \eqref{0.2} into \eqref{3.10}, we have
\begin{align}\label{3.11}\nonumber
\ds\sum_{n=0}^{\infty}B_{9,5}(25n+12)q^n&\equiv 2f_5f^7_{25}\left(\frac{1}{S^7}+\frac{2 q}{S^6}+\frac{2 q^2}{S^5}+\frac{q^3}{S^4}+\frac{2 q^4}{S^3}+\frac{2 q^5}{S^2}+\frac{2 q^6}{S}+q^7+S q^8+2S^2q^9
\right.\\\nonumber
&\qquad\left.\vphantom{\ds\frac{8q^3}{S}}
+ S^3q^{10}+S^4q^{11} +S^5q^{12}+2 S^6q^{13}+2 S^7q^{14}\right)+f^7_{5}f_{25}\left(\ds\frac{1}{S}-q-q^2S\right)q.
\end{align}
If we extract those terms in which the power of $q$ is congruent to 2 modulo 5,  divide by $q^2$, we have
\begin{equation}\label{3.12}
\ds\sum_{n=0}^{\infty}B_{9,5}(125n+62)q^{5n}\equiv 2f_5f^7_{25}\left(2\left(\frac{1}{S^5}-11 q^5-q^{10}S^5\right)+2 q^5\right)+2f^7_{5}f_{25}.
\end{equation}
Substituting \eqref{0.3} into \eqref{3.12}, we have
\begin{align}\label{3.13}\nonumber
\ds\sum_{n=0}^{\infty}B_{9,5}(125n+62)q^{5n}&\equiv 2f_5f^7_{25}\left(2\frac{f^6_5}{f^6_{25}}+2 q^5\right)+2f^7_{5}f_{25},\\
&\equiv f_5f^7_{25}q^5.
\end{align}
The above equation $q^5$ replace by $q$, we have
\begin{equation}\label{3.14}
\ds\sum_{n=0}^{\infty}B_{9,5}(125n+62)q^{n}\equiv f_1f^7_{5}q.
\end{equation}
Substituting \eqref{0.2} into \eqref{3.14}, we have
\begin{equation}\label{3.15}
\ds\sum_{n=0}^{\infty}B_{9,5}(125n+62)q^{n}\equiv f^7_{5}f_{25}\left(\ds\frac{1}{S}-q-q^2S\right)q.
\end{equation}
The equation \eqref{s2} and \eqref{s3} are follow from \eqref{3.15}.
\end{proof}
\ni\textbf{Proof of Theorem \ref{0a}}
From \eqref{s2} using mathematical induction we have \eqref{0a1} and employing \eqref{s3} in \eqref{0a1}, we arrive at the congruence \eqref{0a2}.
\section{Congruence for $(5,11)$-regular bipartition}
Now we will prove Theorem \ref{1a}, we have some results 
\begin{lemma}\label{1}\normalfont For all $n\geq0$ and $m\geq0$,
\begin{equation}\label{1.1a}
B_{5,11}\left(5^{2m}n+\frac{7\cdot5^{2m}-7}{12}\right)\equiv A_mB_{5,11}(5^2n+14)+a_mB_{5,11}(n)\pmod{11},
\end{equation}
where
\begin{equation}\nonumber
\ds A_m=\frac{\sqrt{29}}{29}\left(\frac{1+\sqrt{29}}{2}\right)^{m}-\frac{\sqrt{29}}{29}\left(\frac{1-\sqrt{29}}{2}\right)^{m}
\end{equation}
and
\begin{equation}\nonumber
\ds a_m=\left(\frac{1}{2}-\frac{\sqrt{29}}{58}\right)\left(\frac{1+\sqrt{29}}{2}\right)^{m}
      +\left(\frac{1}{2}+\frac{\sqrt{29}}{58}\right)\left(\frac{1-\sqrt{28}}{2}\right)^{m}.
\end{equation}
\end{lemma}
\begin{proof}
Setting $l=5$ and $m=11$ in \eqref{0.1}, we have
\begin{equation}\label{1.2}
\sum_{n=0}^{\infty}B_{5,11}(n)q^n\equiv f_5f^9_1 \pmod{11}.
\end{equation}
Substituting \eqref{0.2} into \eqref{1.2}, we have
\begin{align}\label{1.3}\nonumber
 \ds\sum_{n=0}^{\infty}B_{5,11}(n)q^n&\equiv f_5f^9_{25}\left(2 q^9+9 q^4 \left(\frac{1}{S^5}-11 q^5-q^{10}S^5\right)+10 q^{18} S^9+2 q^{17} S^8+6 q^{16}S^7
 +10 q^{15} S^6
 \right.\\ \nonumber
&\qquad\left.\vphantom{\ds\frac{8q^3}{S}}
 +5 q^{13} S^4+6 q^{12} S^3+9 q^{11} S^2+7 q^{10} S+\frac{4 q^8}{S}+\frac{9 q^7}{S^2}+\frac{5 q^6}{S^3}+\frac{5 q^5}{S^4}+\frac{10 q^3}{S^6}
 \right.\\ \nonumber
&\qquad\left.\vphantom{\ds\frac{8q^3}{S}}
 +\frac{5 q^2}{S^7}+\frac{2 q}{S^8}+\frac{1}{S^9}\right),\\\nonumber
 &\equiv f_5f^9_{25}\left(2 q^9+9 q^4 \left(\frac{f^6_5}{f^6_{25}}\right)+10 q^{18} S^9+2 q^{17} S^8+6 q^{16}S^7 +10 q^{15} S^6+5 q^{13} S^4+6 q^{12} S^3
 \right.\\ \nonumber
&\qquad\left.\vphantom{\ds\frac{8q^3}{S}}
 +9 q^{11} S^2+7 q^{10} S+\frac{4 q^8}{S}+\frac{9 q^7}{S^2}+\frac{5 q^6}{S^3}+\frac{5 q^5}{S^4}+\frac{10 q^3}{S^6}+\frac{5 q^2}{S^7}+\frac{2 q}{S^8}+\frac{1}{S^9}\right).
 \end{align}
If we extract those terms in which the power of $q$ is congruent to 4 modulo 5, divide by $q^4$ and replace $q^5$ by $q$, we have
\begin{equation}\label{1.4}
\ds\sum_{n=0}^{\infty}B_{5,11}(5n+4)q^n\equiv9f^7_1f^3_5+2f_1f^9_5q.
\end{equation}
Substituting \eqref{0.2} into \eqref{1.4}, we have
\begin{equation}\label{1.5}
\ds\sum_{n=0}^{\infty}B_{5,11}(5n+4)q^n\equiv9f^3_5f^7_{25}\left(\frac{1}{S}-q-q^2S\right)^7+2f^9_5f_{25}\left(\frac{1}{S}-q-q^2S\right)q.
\end{equation}
If we extract those terms in which the power of $q$ is congruent to 2 modulo 5, divide by $q^2$, we have
\begin{eqnarray}\label{1.6}\nonumber
\ds\sum_{n=0}^{\infty}B_{5,11}(5^2n+14)q^{5n}&\equiv&9f^3_5f^7_{25}\left(4q^5+3\left(\frac{1}{S^5}-11 q^5-q^{10}S^5\right)\right)+9f^9_5f_{25},\\
&\equiv&9f^3_5f^7_{25}\left(4q^5+3\left(\frac{f^6_5}{f^6_{25}}\right)\right)+9f^9_5f_{25}.
\end{eqnarray}
The above equation $q^5$ replace by $q$, we have
\begin{equation}\label{1.8}
\ds\sum_{n=0}^{\infty}B_{5,11}(5^2n+14)q^{n}\equiv3f_5f^9_1+3f^7_5f^3_1q.
\end{equation}
Substituting \eqref{0.2} in \eqref{1.8}, we have
\begin{equation}\label{1.13}
\ds\sum^{\infty}_{n=0}B_{5,11}\left(5^2n+14\right)q^n\equiv3f_5f^9_{25}\left(\frac{1}{S}-q-q^2S\right)^9+3f^7_5f^3_{25}\left(\frac{1}{S}-q-q^2S\right)^3q.
\end{equation}
If we extract those terms in which the power of $q$ is congruent to 4 modulo 5, divide by $q^4$, we have
\begin{align}\label{1.14}\nonumber
\ds\sum^{\infty}_{n=0}B_{5,11}\left(5^3n+114\right)q^{5n}&\equiv 3f_5f^9_{25}\left(2q^5+9\left(\frac{1}{S^5}-11q^5-q^{10}S^5\right)\right)+4f^7_5f^3_{25},\\
&\equiv 3f_5f^9_{25}\left(2q^5+9\left(\frac{f^6_5}{f^6_{25}}\right)\right)+4f^7_5f^3_{25}.
\end{align}
The above equation $q^5$ replace by $q$, we have
\begin{equation}\label{1.16}
\ds\sum^{\infty}_{n=0}B_{5,11}\left(5^3n+114\right)q^n\equiv9f^7_1f^3_5+6f_1f^9_5q.
\end{equation}
Substituting \eqref{0.2} in \eqref{1.16} and extract those terms in which the power of $q$ is congruent to 2 modulo 5, divide by $q^2$ and replace $q^5$ by $q$, we have
\begin{equation}\label{1.17}
\sum^\infty_{n=0}B_{5,11}\left(5^4n+364\right)q^n\equiv 10f^9_1f_5+3f^3_1f^7_5q.
\end{equation}
Now compare the equation \eqref{1.17}, \eqref{1.2} and \eqref{1.8}, we have
\begin{equation}\label{1.x}
B_{5,11}(5^4n+364)\equiv B_{5,11}(5^2n+14)+7B_{5,11}(n).
\end{equation}
Which is true if $m=1$. Now, we prove by induction on $k$ in \eqref{1.1a}. Assume that Lemma \ref{1} is true for $k=m-1$ and $k=m$,
\begin{equation}\label{1.17a}
B_{5,11}\left(5^{2(m-1)}n+\frac{7\cdot5^{2(m-1)}-7}{12}\right)\equiv A_{m-1}B_{5,11}(5^2n+14)+a_{m-1}B_{5,11}(n),
\end{equation}
and
\begin{equation}\label{1.17b}
B_{5,11}\left(5^{2m}n+\frac{7\cdot5^{2m}-7}{12}\right)\equiv A_mB_{5,11}(5^2n+14)+a_mB_{5,11}(n).
\end{equation}
Now, we will show that $k=m+1$ is true. Replacing $n=\ds5^{2m-2}+\frac{7\cdot5^{2m-2}-7}{12}$ in \eqref{1.x}, we have
\begin{align}\label{1.17c}\nonumber
B_{5,11}\left(5^{2m+2}n+\frac{7\cdot5^{2m+2}-7}{12}\right)&\equiv B_{5,11}\left(5^{2m}n+\frac{7\cdot5^{2m}-7}{12}\right)\\
&\qquad+7B_{5,11}\left(5^{2m-2}n+\frac{7\cdot5^{2m-2}-7}{12}\right).
\end{align}
we can easily check that
\begin{equation}\label{1.17d}
A_{m+1}=A_m+7A_{m-1}
\end{equation}
and
\begin{equation}\label{1.17e}
a_{m+1}=a_m+7a_{m-1}
\end{equation}
From equation \eqref{1.17a}, \eqref{1.17b}, \eqref{1.17c}, \eqref{1.17d} and \eqref{1.17e}, we have
\begin{align}\label{1.17f}\nonumber
B_{5,11}\left(5^{2m+2}n+\frac{7\cdot5^{2m+2}-7}{12}\right)&\equiv A_mB_{5,11}(5^2n+14)+a_mB_{5,11}(n)\\\nonumber
&\qquad+7(A_{m-1}B_{5,11}(5^2n+14)+a_{m-1}B_{5,11}(n)),\\\nonumber
&\equiv(A_m+7A_{m-1})B_{5,11}(5^2n+14)+(a_m+7a_{m-1})B_{5,11}(n),\\
&\equiv A_{m+1}B_{5,11}(5^2n+14)+a_{m+1}B_{5,11}(n).
\end{align}
Therefore the Lemma \ref{1} is true by induction of \eqref{1.1a}.
\end{proof}
\begin{proposition}\label{2}\normalfont For all $n\geq0$,
\begin{eqnarray}
\label{1.18a}
B_{5,11}\left(5^{11}(5n+k)+\frac{11\cdot5^{11}-7}{12}\right)&\equiv&0\pmod{11},~~\textrm{where}~~ k=4,5.\\
\label{1.18b}
B_{5,11}\left(5^{12}n+\frac{7\cdot5^{12}-7}{12}\right)&\equiv&2B_{5,11}(n)\pmod{11}.
\end{eqnarray}
\end{proposition}
\begin{proof}
From Lemma \ref{1}, put $m=5$, we find that, $A_5\equiv5\pmod{11}$ and $a_{5}\equiv6\pmod{11}$,
\begin{equation}\label{1.19b}
\ds\sum_{n=0}^{\infty}B_{5,11}\left(5^{10}n+\frac{7\cdot5^{10}-7}{12}\right)q^n\equiv 10f^9_1f_5+4f^3_1f^7_5q.
\end{equation}
Substituting \eqref{0.2} into \eqref{1.19b}, we have
\begin{equation}\label{1.19c}\nonumber
\ds\sum_{n=0}^{\infty}B_{5,11}\left(5^{10}n+\frac{7\cdot5^{10}-7}{12}\right)q^n\equiv10f_5f^9_{25}\left(\frac{1}{S}-q-q^2S\right)^9
+4f^7_5f^3_{25}\left(\frac{1}{S}-q-q^2S\right)^3q.
\end{equation}
If we extract those terms in which the power of $q$ is congruent to 4 modulo 5, divide by $q^4$, we have
\begin{align}\label{1.19d}\nonumber
\ds\sum_{n=0}^{\infty}B_{5,11}\left(5^{11}n+\frac{11\cdot5^{11}-7}{12}\right)q^{5n}
&\equiv10f_5f^9_{25}\left(2q^5+9\left(\frac{1}{S^5}-11 q^5-q^{10}S^5\right)\right)
+9f^7_5f^3_{25},\\
&\equiv10f_5f^9_{25}\left(2q^5+9\left(\frac{f^6_5}{f^6_{25}}\right)\right)
+9f^7_5f^3_{25}.
\end{align}
The above equation $q^5$ replace by $q$, we have
\begin{equation}\label{1.19e}
\ds\sum_{n=0}^{\infty}B_{5,11}\left(5^{11}n+\frac{11\cdot5^{11}-7}{12}\right)q^n\equiv9f_1f^9_5q.
\end{equation}
Substituting \eqref{0.2} into \eqref{1.19e}, we have
\begin{equation}\label{1.19f}
\ds\sum^{\infty}_{n=0}B_{5,11}\left(5^{11}n+\frac{11\cdot5^{11}-7}{12}\right)q^n\equiv9f^9_5f_{25}\left(\frac{1}{S}-q-q^2S\right)q.
\end{equation}
This  completes  the  proof of equation \eqref{1.18a} and \eqref{1.18b} follow from \eqref{1.19f}.
\end{proof}
\ni\textbf{Proof of Theorem \ref{1a}.} Equation \eqref{12} follow from \eqref{1.18b} by mathematical induction. Empolying \eqref{1.18a} in \eqref{12}, we obtain \eqref{13}.
\section{Congruence for $(5,13)$-regular bipartition}
In this section we prove Theorem \ref{1b}, we shell develop some basic results.
\begin{lemma}\normalfont\label{3} For all $n\geq0$ and $m\geq0$,
\begin{equation}\label{2.1}
B_{5,13}\left(5^{2m}n+\frac{2\cdot5^{2m}-2}{3}\right)\equiv C_mB_{5,13}(5^2n+16)+c_mB_{5,13}(n)\pmod{13},
\end{equation}
where
\begin{equation}
\ds C_m=\frac{\sqrt{17}}{34}\left(4+\sqrt{17}\right)^{m}-\frac{\sqrt{17}}{34}\left(4-\sqrt{17}\right)^{m}
\end{equation}
and
\begin{equation}
\ds c_m=\left(\frac{1}{2}-\frac{2\sqrt{17}}{17}\right)\left(4+\sqrt{17}\right)^{m}
      +\left(\frac{1}{2}+\frac{2\sqrt{17}}{17}\right)\left(4-\sqrt{17}\right)^{m}
\end{equation}
\end{lemma}
\begin{proof}
Setting $l=5$ and $m=13$ in \eqref{0.1}, we have
\begin{equation}\label{2.2}
\sum_{n=0}^{\infty}B_{5,13}(n)q^n\equiv f_5f^{11}_1 \pmod{13}.
\end{equation}
Substituting \eqref{0.2} in \eqref{2.2}, we have
\begin{align}\label{2.3}\nonumber
 \ds\sum_{n=0}^{\infty}B_{5,13}(n)q^n&\equiv f_5f^{11}_{25}\left(\vphantom{\ds\frac{8q^3}{S}}12 q^{22} S^{11}+2 q^{21} S^{10}+8 q^{20} S^9+10 q^{19} S^8+6 q^{18} S^7+12 q^{17} S^6
 \right.\\ \nonumber
&\qquad\left.\vphantom{\ds\frac{8q^3}{S}}
 +7 q^{16} S^5+12 q^{14} S^3+4 q^{13} S^2+10 q^{12} S+8 q^{11}+\frac{3 q^{10}}{S}+\frac{4 q^9}{S^2}+\frac{q^8}{S^3}
 \right.\\ \nonumber
&\qquad\left.\vphantom{\ds\frac{8q^3}{S}}
 +\frac{6 q^6}{S^5}+\frac{12 q^5}{S^6}+\frac{7 q^4}{S^7}+\frac{10 q^3}{S^8}+\frac{5 q^2}{S^9}+\frac{2 q}{S^{10}}+\frac{1}{S^{11}}\right),\\\nonumber
&\equiv f_5f^{11}_{25}\left(8q^{11}+2 q \left(\frac{1}{S^5}-11 q^5-q^{10}S^5\right)^2+11 q^6 \left(\frac{1}{S^5}-11 q^5-q^{10}S^5\right)
\right.\\ \nonumber
&\qquad\left.\vphantom{\ds\frac{8q^3}{S}}
+12 q^{22} S^{11}+8 q^{20} S^9+10 q^{19} S^8+6 q^{18} S^7+12 q^{17} S^6+12 q^{14} S^3+4 q^{13} S^2
\right.\\ \nonumber
&\qquad\left.\vphantom{\ds\frac{8q^3}{S}}
+10 q^{12} S+\frac{3 q^{10}}{S}+\frac{4 q^9}{S^2}+\frac{q^8}{S^3}+\frac{12 q^5}{S^6}+\frac{7 q^4}{S^7}+\frac{10 q^3}{S^8}+\frac{5 q^2}{S^9}+\frac{1}{S^{11}}\right),\\\nonumber
&\equiv f_5f^{11}_{25}\left(8q^{11}+2 q \left(\frac{f^6_5}{f^6_{25}}\right)^2+11 q^6 \left(\frac{f^6_5}{f^6_{25}}\right)
+12 q^{22} S^{11}+8 q^{20} S^9+10 q^{19} S^8
\right.\\ \nonumber
&\qquad\left.\vphantom{\ds\frac{8q^3}{S}}
+6 q^{18} S^7+12 q^{17} S^6+12 q^{14} S^3+4 q^{13} S^2+10 q^{12} S+\frac{3 q^{10}}{S}+\frac{4 q^9}{S^2}+\frac{q^8}{S^3}
\right.\\
&\qquad\left.\vphantom{\ds\frac{8q^3}{S}}
+\frac{12 q^5}{S^6}+\frac{7 q^4}{S^7}+\frac{10 q^3}{S^8}+\frac{5 q^2}{S^9}+\frac{1}{S^{11}}\right).
\end{align}
If we extract those terms in which the power of $q$ is congruent to 1 modulo 5, divide by $q$ and replace $q^5$ by $q$, we have
\begin{equation}\label{2.4}
\ds\sum_{n=0}^{\infty}B_{5,13}(5n+1)q^n\equiv2\frac{f_{13}}{f_5}+11f^7_1f^5_5q+8f_1f^{11}_5q^2.
\end{equation}
Substituting \eqref{0.2} in \eqref{2.4}, we have
\begin{align}\label{2.5}\nonumber
\ds\sum_{n=0}^{\infty}B_{5,13}(5n+1)q^n&\equiv2\frac{f_{325}}{f_5}\left(\frac{1}{S_1}-q^{13}-q^{26}S_1\right)+11f^5_5f^7_{25}\left(\frac{1}{S}-q-q^2S\right)^7q\\
&+8f^{11}_5f_{25}\left(\frac{1}{S}-q-q^2S\right)q^2.
\end{align}
If we extract those terms in which the power of $q$ is congruent to 3 modulo 5, divide by $q^3$, we have
\begin{equation}\label{2.6}\nonumber
\ds\sum_{n=0}^{\infty}B_{5,13}(5^2n+16)q^{5n}\equiv11\frac{f_{325}}{f_5}q^{10}+11f^5_5f^7_{25}\left(8q^5+\left(\frac{1}{S^5}-11 q^5-q^{10}S^5\right)\right)+5f^{11}_5f_{25},
\end{equation}
from which can be written as
\begin{align}\label{2.7}\nonumber
\ds\sum_{n=0}^{\infty}B_{5,13}(5^2n+16)q^{5n}&\equiv3f^{11}_5f_{25}+10f^5_5f^7_{25}q^5+11\frac{f_{325}}{f_5}q^{10}.
\end{align}
The above equation $q^5$ replace by $q$, we have
\begin{equation}\label{2.8}
\ds\sum_{n=0}^{\infty}B_{5,13}(5^2n+16)q^{n}\equiv3f^{11}_1f_5+10f^5_1f^7_{5}q+11\frac{f_{65}}{f_1}q^2.
\end{equation}
Substituting \eqref{0.2} and \eqref{0.3a} in \eqref{2.8}, we have
\begin{align}\label{2.9}\nonumber
\ds\sum_{n=0}^{\infty}B_{5,13}\left(5^2n+16\right)q^{n}&\equiv 3f_5f^{11}_{25}\left(\frac{1}{S}-q-q^2S\right)^{11}
+10f^7_{5}f^5_{25}\left(\frac{1}{S}-q-q^2S\right)^5q\\\nonumber
&\qquad+11\frac{f_{65}f^5_{25}}{f_5^6}\left(\vphantom{\ds\frac{8q^3}{S}}q^8 S^4-q^7 S^3+2 q^6 S^2-3 q^5 S+5 q^4
\right.\\
&\qquad\qquad\left.\vphantom{\ds\frac{8q^3}{S}}
+\frac{3 q^3}{S}+\frac{2 q^2}{S^2}+\frac{q}{S^3}+\frac{1}{S^4}\right)q^2.
\end{align}
If we extract those terms in which the power of $q$ is congruent to 1 modulo 5, divide by $q$, we have
\begin{align}\label{2.10}\nonumber
\ds\sum_{n=0}^{\infty}B_{5,13}\left(5^3n+41\right)q^{5n}&\equiv 3f_5f^{11}_{25}\left(8q^{10}+2\left(\frac{f^6_5}{f^6_{25}}\right)^2
+11q^5 \left(\frac{f^6_5}{f^6_{25}}\right)\right)+10f^7_{5}f^5_{25}\left(\frac{f^6_5}{f^6_{25}}\right)\\
&\qquad+3\frac{f_{65}f^5_{25}}{f_5^6}q^5.
\end{align}
The above equation $q^5$ replace by $q$, we have
\begin{equation}\label{2.11}
\ds\sum_{n=0}^{\infty}B_{5,13}\left(5^3n+41\right)q^{n}\equiv3\frac{f_{13}}{f_5}+10f^7_1f^5_5q+11f_1f^{11}_5q^2.
\end{equation}
Substituting \eqref{0.2} in \eqref{2.11}, we have
\begin{align}\label{2.12}\nonumber
\ds\sum_{n=0}^{\infty}B_{5,13}\left(5^3n+41\right)q^n&\equiv3\frac{f_{325}}{f_5}\left(\frac{1}{S_1}-q^{13}-q^{26}S_1\right)
+10f^5_5f^7_{25}\left(\frac{1}{S}-q-q^2S\right)^7q\\
&\qquad+11f^{11}_5f_{25}\left(\frac{1}{S}-q-q^2S\right)q^2.
\end{align}
If we extract those terms in which the power of $q$ is congruent to 3 modulo 5, divide by $q^3$, we have
\begin{equation}\label{2.13}
\ds\sum_{n=0}^{\infty}B_{5,13}\left(5^4n+416\right)q^{5n}\equiv10\frac{f_{325}}{f_5}q^{10}
+10f^5_5f^7_{25}\left(8q^5+\left(\frac{f^6_5}{f^6_{25}}\right)\right)+2f^{11}_5f_{25}
\end{equation}
The above equation $q^5$ replace by $q$, we have
\begin{equation}\label{2.14}
\ds\sum_{n=0}^{\infty}B_{5,13}(5^4n+416)q^{n}\equiv12f^{11}_1f_5+2f^5_1f^7_{5}q+10\frac{f_{65}}{f_1}q^2.
\end{equation}
Now compare the equation \eqref{2.14}, \eqref{2.2} and \eqref{2.8}, we have.
\begin{equation}\label{2.x}
B_{5,13}(5^4n+416)\equiv8B_{5,13}(5^2n+16)+B_{5,13}(n).
\end{equation}
Which is true if $m=1$. Now, we prove by induction on $k$ in \eqref{2.1}. Assume that Lemma \ref{3} is true for $k=m-1$ and $k=m$,
\begin{equation}\label{2.14a}
B_{5,13}\left(5^{2m-2}n+\frac{2\cdot5^{2m-2}-2}{3}\right)\equiv C_{m-1}B_{5,13}(5^2n+16)+c_{m-1}B_{5,13}(n)
\end{equation}
and
\begin{equation}\label{2.14b}
B_{5,13}\left(5^{2m}n+\frac{2\cdot5^{2m}-2}{3}\right)\equiv C_{m}B_{5,13}(5^2n+16)+c_{m}B_{5,13}(n).
\end{equation}
Now, we will show that $k=m+1$ is true. Replacing $n=\ds5^{2m-2}+\frac{2\cdot5^{2m-2}-2}{3}$ in \eqref{2.x}, we have
\begin{align}\label{2.14c}\nonumber
B_{5,13}\left(5^{2m+2}n+\frac{2\cdot5^{2m+2}-2}{3}\right)&\equiv8B_{5,13}\left(5^{2m}n+\frac{2\cdot5^{2m}-2}{3}\right)\\
&\qquad+B_{5,13}\left(5^{2m-2}n+\frac{2\cdot5^{2m-2}-2}{3}\right).
\end{align}
we can easily check that
\begin{equation}\label{2.14d}
C_{m+1}=8C_m+C_{m-1}
\end{equation}
and
\begin{equation}\label{2.14e}
c_{m+1}=8c_m+c_{m-1}
\end{equation}
From equation \eqref{2.14a}, \eqref{2.14b}, \eqref{2.14c}, \eqref{2.14d} and \eqref{2.14e}, we have
\begin{align}\label{2.14f}\nonumber
B_{5,13}\left(5^{2m+2}n+\frac{2\cdot5^{2m+2}-2}{3}\right)&\equiv8(C_mB_{5,13}(5^2n+16)+c_mB_{5,13}(n))\\\nonumber
&\qquad+C_{m-1}B_{5,11}(5^2n+16)+c_{m-1}B_{5,11}(n),\\\nonumber
&\equiv(8C_m+C_{m-1})B_{5,13}(5^2n+16)+(8c_m+c_{m-1})B_{5,13}(n),\\
&\equiv C_{m+1}B_{5,13}(5^2n+16)+c_{m+1}B_{5,13}(n).
\end{align}
Therefore the Lemma \ref{3} is true by induction of \eqref{2.1}.
\end{proof}
\begin{proposition}\normalfont For all $n\geq0$,
\begin{eqnarray}
\label{2.16a}
B_{5,13}\left(5^6n+\frac{(3k+1)\cdot5^5-2}{3}\right)&\equiv&0\pmod{13},~~\textrm{where}~~ k=1,5.\\
\label{2.16b}
B_{5,13}\left(5^6n+\frac{2\cdot5^6-2}{3}\right)&\equiv&8B_{5,13}(n)\pmod{13}.
\end{eqnarray}
\end{proposition}
\begin{proof}
From Lemma \ref{2}, put $m=2$, we find that, $C_2\equiv8\pmod{13}$ and $c_2\equiv1\pmod{13}$,
Now substituting \eqref{0.2} and \eqref{0.3a} into \eqref{2.14}, we have
\begin{align}\label{2.17}\nonumber
\ds\sum_{n=0}^{\infty}B_{5,13}(5^4n+416)q^{n}&\equiv12f_5f^{11}_{25}\left(\frac{1}{S}-q-q^2S\right)^{11}
+2f^7_{5}f^5_{25}\left(\frac{1}{S}-q-q^2S\right)^5q\\\nonumber
&\qquad+10\frac{f_{65}f^5_{25}}{f_5^6}\left(\vphantom{\ds\frac{8q^3}{S}}q^8 S^4-q^7 S^3+2 q^6 S^2-3 q^5 S+5 q^4
\right.\\
&\qquad\qquad\left.\vphantom{\ds\frac{8q^3}{S}}
+\frac{3 q^3}{S}+\frac{2 q^2}{S^2}+\frac{q}{S^3}+\frac{1}{S^4}\right)q^2.
\end{align}
If we extract those terms in which the power of $q$ is congruent to 1 modulo 5, divide by $q$, we have
\begin{align}\label{2.19}\nonumber
\ds\sum_{n=0}^{\infty}B_{5,13}(5^5n+1041)q^{5n}&\equiv12f_5f^{11}_{25}\left(8q^{10}+2\left(\frac{f^6_5}{f^6_{25}}\right)^2
+11q^5 \left(\frac{f^6_5}{f^6_{25}}\right)\right)\\
&+2f^7_{5}f^5_{25}\left(\frac{f^6_5}{f^6_{25}}\right)+11\frac{f_{65}f^5_{25}}{f_5^6}q^5\pmod{13}.
\end{align}
The above equation $q^5$ replace by $q$, we have
\begin{equation}\label{2.20}
\ds\sum_{n=0}^{\infty}B_{5,13}(5^5n+1041)q^{n}\equiv5f_1f^{11}_5q^2.
\end{equation}
Substituting \eqref{0.2} into \eqref{2.20}, we have
\begin{equation}\label{2.21}
\ds\sum_{n=0}^{\infty}B_{5,13}(5^5n+1041)q^{n}\equiv5f^{11}_5f_{25}\left(\frac{1}{S}-q-q^2S\right)q^2.
\end{equation}
This  completes  the  proof of equation \eqref{2.16a} and \eqref{2.16b}, follow from \eqref{2.21}.
\end{proof}
\ni\textbf{Proof of Theorem \ref{1b}.} Equation \eqref{14} follow from \eqref{2.16b} by mathematical induction. Employing \eqref{2.16a} in \eqref{14}, we obtain \eqref{15}.
\section{Congruence for $(81,17)$-regular bipartition}
In this section we shall develop some results to prove our main Theorem \ref{1c}.
\begin{proposition}\normalfont For $n\geq0$, we have
\begin{equation}\label{s8}
B_{81,17}(27(3n+k)+23)\equiv 0\pmod{17},~~ \textrm{where}~~ k\in\{2,3\}
\end{equation}
and
\begin{equation}\label{s9}
\ds\sum_{n=0}^{\infty}B_{81,17}(81n+50)q^n\equiv5f^{16}_1\pmod{17}.
\end{equation}
\end{proposition}
\begin{proof}
Setting $l=81, m=17$ in \eqref{0.1}, we have
\begin{equation}\label{7.1}
\ds\sum_{n=0}^{\infty}B_{81,17}(n)q^n=\frac{f_{81}f_{17}}{f_1^2}.
\end{equation}
From \eqref{k1} and \eqref{7.1}, we have
\begin{equation}\label{7.2}
\ds\sum_{n=0}^{\infty}B_{81,17}(n)q^n\equiv f_{81}f^{15}_1 \pmod{17}.
\end{equation}
Entry $1(\textrm{iv})$ on page $345$ of \cite{Berndt} is Ramanujan's cubic continued fraction
\begin{align}\label{7.3}
f^3_1 &=f^3_9(u^{-1}-3q+4q^3u^2),
\end{align}
where
\begin{equation}\nonumber
u=\ds\frac{f_3f^3_{18}}{f_6f^3_9}.
\end{equation}
Substituting \eqref{7.3} into \eqref{7.2}, we have
\begin{align}\label{7.4}\nonumber
\ds\sum_{n=0}^{\infty}B_{81,17}(n)q^n&\equiv f_{81}f^{15}_9 \left(\frac{1}{u^5}+\frac{2 q}{u^4}+\frac{5 q^2}{u^3}+\frac{5 q^3}{u^2}+\frac{12 q^4}{u}+4 q^5+6 q^6 u+10 q^7 u^2+2 q^8 u^3+9 q^9 u^4
\right.\\
&\qquad\left.\vphantom{\ds\frac{8q^3}{S}}
+2 q^{10} u^5+14 q^{11} u^6+5 q^{12} u^7+2 q^{13} u^8+4 q^{15} u^{10}\right).
\end{align}
If we extract those terms in which the power of $q$ is congruent to 2 modulo 3, divide by $q^2$ and replace $q^3$ by $q$, we have
\begin{align}\label{7.5}\nonumber
\ds\sum_{n=0}^{\infty}B_{81,17}(3n+2)q^n&\equiv f_{27}f^{15}_3\left(\frac{5}{v^3}+4 q+2 q^2 v^3+14 q^{3} v^6\right),\\
&\equiv f_{27}f^{15}_3\left(5\left(\frac{1}{v}+4 q v^2\right)^3+12 q\right).
\end{align}
Here we have used Entry $1$ on page $345$ of in \cite{Berndt}, namely,\\
\begin{align}\label{7.6}
\ds\frac{f^{12}_1}{f^{12}_3}+27q&=(v^{-1}+4qv^2)^3,
\end{align}
where
\begin{equation}\nonumber
v:=\ds\frac{f_1f^3_6}{f_2f^3_3}.
\end{equation}
Substituting \eqref{7.6} into \eqref{7.5}, we have
\begin{align}\label{7.7}\nonumber
\ds\sum_{n=0}^{\infty}B_{81,17}(3n+2)q^n&\equiv f_{27}f^{15}_3\left(5 \left(\frac{f_1^{12}}{f_3^{12}}+27 q\right)+12 q\right),\\
&\equiv 5 f_1^{12} f_{27} f_3^3+11 f_{27} f_3^{15} q.
\end{align}
Substituting \eqref{7.3} into \eqref{7.7}, we have
\begin{align}\label{7.8}\nonumber
\ds\sum_{n=0}^{\infty}B_{81,17}(3n+2)q^n&\equiv 5 f_9^{12} f_{27} f_3^3\left(\frac{1}{u^4}+\frac{5 q}{u^3}+\frac{3 q^2}{u^2}+\frac{10 q^3}{u}+5 q^4+7 q^5 u+4 q^6 u^2+2 q^7 u^3+14 q^8 u^4
\right.\\
&\qquad\left.\vphantom{\ds\frac{8q^3}{S}}
+q^9 u^5+14 q^{10} u^6+q^{12} u^8\right)+11 f_{27} f_3^{15} q.
\end{align}
If we extract those terms in which the power of $q$ is congruent to 1 modulo 3, divide by $q$ and replace $q^3$ by $q$, we have
\begin{align}\label{7.9}\nonumber
\ds\sum_{n=0}^{\infty}B_{81,17}(9n+5)q^n&\equiv 5 f_3^{12} f_{9} f_1^3\left(\frac{5}{v^3}+5 q+2 q^2 v^3+14 q^{3} v^6\right)+11 f_{9} f_1^{15},\\ &\equiv 5 f_3^{12} f_{9} f_1^3\left(5 \left(\frac{1}{v}+4 q v^2\right)^3+13 q\right)+11 f_{9} f_1^{15}.
\end{align}
Substituting \eqref{7.6} into \eqref{7.9}, we have
\begin{align}\label{7.10}\nonumber
\ds\sum_{n=0}^{\infty}B_{81,17}(9n+5)q^n&\equiv5 f_3^{12} f_{9} f_1^3\left(5 \left(\frac{f_1^{12}}{f_3^{12}}+27 q\right)+13 q\right)+11 f_{9} f_1^{15},\\
&\equiv 2 f_9 f_1^{15}+9 f_3^{12} f_9 f_1^3 q.
\end{align}
Substituting \eqref{7.3} into \eqref{7.10}, we have
\begin{align}\label{7.11}\nonumber
\ds\sum_{n=0}^{\infty}B_{81,17}(9n+5)q^n&\equiv2 f_9^{16}\left(\frac{1}{u^5}+\frac{2 q}{u^4}+\frac{5 q^2}{u^3}+\frac{5 q^3}{u^2}+\frac{12 q^4}{u}+4 q^5+6 q^6 u+10 q^7 u^2+2 q^8 u^3+9 q^9 u^4
\right.\\\nonumber
&\qquad\qquad\left.\vphantom{\ds\frac{8q^3}{S}}
+2 q^{10} u^5+14 q^{11} u^6+5 q^{12} u^7+2 q^{13} u^8+4 q^{15} u^{10}\right)\\
&\qquad+9 f_3^{12}f_9^4\left(qu^{-1}-3q^2+4q^4u^2\right).
\end{align}
If we extract those terms in which the power of $q$ is congruent to 2 modulo 3, divide by $q^2$ and replace $q^3$ by $q$, we have
\begin{align}\label{7.12}\nonumber
\ds\sum_{n=0}^{\infty}B_{81,17}(27n+23)q^n&\equiv2 f_3^{16}\left(\frac{5}{v^3}+4 q+2 q^2 v^3+14 q^{3} v^6\right)+7 f_1^{12}f_3^4,\\
&\equiv2 f_3^{16}\left(5\left(\frac{1}{v}+4 q v^2\right)^3+12 q\right)+7 f_1^{12}f_3^4.
\end{align}
Substituting \eqref{7.6} into \eqref{7.12}, we have
\begin{align}\label{7.13}\nonumber
\ds\sum_{n=0}^{\infty}B_{81,17}(27n+23)q^n&\equiv 2 f_3^{16}\left(5 \left(\frac{f_1^{12}}{f_3^{12}}+27 q\right)+12 q\right)+7 f_1^{12}f_3^4,\\
&\equiv 5 f_3^{16} q.
\end{align}
This  completes  the  proof of equation \eqref{s8} and \eqref{s9}, follow from \eqref{7.13}.
\end{proof}
\begin{corollary}\label{q}\normalfont For $n\geq0$, we have
\begin{equation}\label{s10}
b_{17}\left(4^9n+\frac{2(4^8-1)}{3}\right)\equiv0\pmod{17},
\end{equation}
\begin{equation}\label{s11}
b_{17}\left(4^9n+\frac{2(4^9-1)}{3}\right)\equiv8f^{16}_1\pmod{17}
\end{equation}
and
\begin{equation}\label{s12}
b_{17}\left(2\cdot4^8+\frac{5\cdot4^8-2)}{3}\right)\equiv b_{17}(2n+1)\pmod{17}
\end{equation}
\end{corollary}
\begin{proof}
In \cite[Lemma 4.1]{Xia}, we have, for all $n\geq0$,
\begin{equation}\label{7.14}
b_{17}\left(4^kn+\frac{2(4^k-1)}{3}\right)\equiv D(k)b_{17}(4n+2)+d(k)b_{17}(n) \pmod{17},
\end{equation}
where
\begin{eqnarray}\nonumber
D(k)&=&\frac{\sqrt{5}}{10}\left((1+\sqrt{5})^k-(1-\sqrt{5})^k\right), \\\nonumber
d(k)&=&\left(\frac{1}{2}-\frac{\sqrt{5}}{10}\right)(1+\sqrt{5})^k+\left(\frac{1}{2}+\frac{\sqrt{5}}{10}\right)(1-\sqrt{5})^k.
\end{eqnarray}
Put $k=8$ in \eqref{7.14}, we can easily find that $C(8)\equiv2 \pmod{17}$ and $c(8)\equiv13 \pmod{17}$, therefore
\begin{align}\label{7.15}\nonumber
b_{17}\left(4^8n+\frac{2(4^8-1)}{3}\right)&\equiv C(8)b_{17}(4n+2)+c(8)b_{17}(n),\\
&\equiv 2b_{17}(4n+2)+13b_{17}(n).
\end{align}
where
\begin{equation}\label{7.16}
\ds\sum_{n=0}^{\infty}b_{17}(n)q^n \equiv f^{16}_1\pmod{17}
\end{equation}
and
\begin{align}\label{7.17}
\ds\sum_{n=0}^{\infty}b_{17}(4n+2)q^n &\equiv2f^{16}_1+9q\frac{f^{24}_2}{f^8_1}.
\end{align}
Substituting \eqref{7.16} and \eqref{7.17} into \eqref{7.15}, we have
\begin{align}\label{7.18}\nonumber
\ds\sum_{n=0}^{\infty}b_{17}\left(4^8n+\frac{2(4^8-1)}{3}\right)q^n&\equiv2\left(2f^{16}_1+9q\frac{f^{24}_2}{f^8_1}\right)+13f^{16}_1,\\
&\equiv q\frac{f^{24}_2}{f^8_1}.
\end{align}
Substituting \eqref{e2} into \eqref{7.18}, we have
\begin{equation}\label{7.19}
\ds\sum_{n=0}^{\infty}b_{17}\left(4^8n+\frac{2(4^8-1)}{3}\right)q^n\equiv qf^{24}_2\left(\frac{f_4^{14}}{f^{14}_2f_8^4}+4q\frac{f_4^2f_8^4}{f^{10}_2}\right)^2.
\end{equation}
From \eqref{7.19}, we have
\begin{equation}\label{7.20}
\ds\sum_{n=0}^{\infty}b_{17}\left(2\cdot4^8n+\frac{2(4^8-1)}{3}\right)q^n\equiv8qf^{16}_2
\end{equation}
and
\begin{equation}\label{7.21}
\ds\sum_{n=0}^{\infty}b_{17}\left(4^8(2n+1)+\frac{2(4^8-1)}{3}\right)q^n\equiv \frac{f^{20}_2}{f^4_1f^8_4}+16qf^4_1f^4_2f^8_4.
\end{equation}
This  completes  the  proof of equation \eqref{s10}, \eqref{s11}  and \eqref{s12}, follow from \eqref{7.20} and \eqref{7.21}.
\end{proof}
\ni\textbf{Proof of Theorem \ref{1c}.}
From equation \eqref{s9} and \eqref{7.16}, we have
\begin{equation}\label{7.22}
B_{81,17}(81n+50)\equiv5b_{17}(n)
\end{equation}
Now, replacing $\ds n=4^{9m}n+\frac{2\cdot4^{8m}-2}{3}$, $\ds n=4^{9m}n+\frac{2\cdot4^{9m}-2}{3}$ and $\ds n=2\cdot4^{8m}n+\frac{5\cdot4^{8m}-2}{3}$ in \eqref{7.22} and employing with Corollary \ref{q}, 
we have \eqref{s13}, \eqref{s14} and \eqref{s15}.
\section{Proof of Theorem \ref{1d}.}
Setting $l=2$ and $m=8$ in \eqref{0.1}, we have
\begin{equation}\label{5.1}
\sum_{n=0}^{\infty}B_{2,8}(n)q^n=\frac{f_2f_8}{f_1^2}.
\end{equation}
Substituting \eqref{a} into \eqref{5.1}, we have
\begin{equation}\label{5.2}
\sum_{n=0}^{\infty}B_{2,8}(n)q^n=f_2f_8\left(\frac{f^5_8}{f^5_2f^2_{16}}+2q\frac{f^2_4f^2_{16}}{f_2^5f_8}\right).
\end{equation}
If we extract those terms in which the power of $q$ is congruent to 1 modulo 2, divide by $q$ and replace $q^2$ by $q$, we have
\begin{equation}\label{5.3}
\sum_{n=0}^{\infty}B_{2,8}(2n+1)q^n=2\frac{f^2_2f^2_8}{f^4_1}.
\end{equation}
Substituting \eqref{b} into \eqref{5.3}, we have
\begin{equation}\label{5.4}
\sum_{n=0}^{\infty}B_{2,8}(2n+1)q^n=2f^2_2f^2_8\left(\frac{f_4^{14}}{f^{14}_2f_8^4}+4q\frac{f_4^2f_8^4}{f^{10}_2}\right).
\end{equation}
If we extract those terms in which the power of $q$ is congruent to 1 modulo 2, divide by $q$ and replace $q^2$ by $q$, we have
\begin{equation}\label{5.5}
\sum_{n=0}^{\infty}B_{2,8}(4n+3)q^n=8\frac{f^2_2f^6_4}{f^8_1}.
\end{equation}
Substituting \eqref{b} into \eqref{5.5}, we have
\begin{eqnarray}\label{5.6}\nonumber
\sum_{n=0}^{\infty}B_{2,8}(4n+3)q^n&=&8f^2_2f^6_4\left(\frac{f_4^{14}}{f^{14}_2f_8^4}+4q\frac{f_4^2f_8^4}{f^{10}_2}\right)^2,\\
&=&8f^2_2f^6_4\left(\frac{f_4^{28}}{f^{28}_2f_8^8}+8q\frac{f_4^{16}}{f^{24}_2}+16q^2\frac{f_4^4f_8^8}{f^{20}_2}\right).
\end{eqnarray}
If we extract those terms in which the power of $q$ is congruent to 1 modulo 2, divide by $q$ and replace $q^2$ by $q$, we have
\begin{equation}\label{5.7}
\sum_{n=0}^{\infty}B_{2,8}(8n+7)q^n=64\frac{f^{22}_2}{f^{22}_1}.
\end{equation}
From binomial theorem, we have
\begin{equation}\label{5.7}
\sum_{n=0}^{\infty}B_{2,8}(8n+7)q^n\equiv9\frac{f^{2}_{22}}{f^{2}_{11}}\pmod{11}.
\end{equation}
This  completed  the  proof Theorem \ref{1d}, follow from \eqref{5.7}.
\subsection*{Acknowledgment}
I would like to thank Prof. K. Srinivas for some useful discussions which improve the presentation of the paper. This work is party supported by SERB MATRIX project No. MTR/2017/001006.


\begin{thebibliography}{99}

\bibitem{Adiaga}
 C. Adiaga and D. Ranganatha, A simple proof of a conjecture of Dou on $(3,7)$-regular bipartitions modulo $3$, Integers $17$ $(2017)$.

 \bibitem{Andrews}
G.E. Andrews, The Theory of Partitions, Cambridge Univ. Press, Cambridge, $1998$.

\bibitem{Andrews1}
G. E. Andrews, M. D. Hirschhorn and J. A. Sellers, Arithmetic properties of partitions with even parts distinct, Ramanujan J. $23$, $(2010)$ $169–181$.


\bibitem{Berndt}
B.C. Berndt, Ramanujan’s Notebooks, Part III, Springer, New York, 1991.

\bibitem{Carlson}
R. Carlson and J.J. Webb, Infinite families of congruences for $k$-regular partitions, Ramanujan J., $33$, $(2014)$ $329–337$.

\bibitem{Cui}
S.P. Cui and N.S.S. Gu, Arithmetic properties of the $\ell$-regular partitions, Adv. Appl. Math., $51$, $(2013)$ $507–523$.

 \bibitem{Cui1}
S.P. Cui and N.S.S. Gu, Congruences for $9$-regular partitions modulo $3$, Ramanujan J., $38$, $(2015)$ $503-512$.

\bibitem{Dandurand}
B. Dandurand and D. Penniston, $\ell$-divisibility of $\ell$-regular partition functions, Ramanujan J., $19$, $(2009)$ $63–70$.

\bibitem{Dou}
D.Q.J. Dou, Congruences for $(3, 11)$-regular bipartitions modulo 11, Ramanujan J., $40$, $(2016)$ $535–540$.

\bibitem{Furcy}
D. Furcy and D. Penniston, Congruences for $\ell$-regular partition functions modulo $3$, Ramanujan J., $27$, $(2012)$ $101–108$.

\bibitem{Gordon}
B. Gordon and K. Ono, Divisibility of certain partition functions by powers of primes, Ramanujan J., $1$, $(1997)$ $25–34$.


\bibitem{Hirschhorn}
M.D. Hirschhorn and J.A. Sellers, Elementary proofs of parity results for $5$-regular partitions, Bull. Aust. Math. Soc., $81$, $(2010)$ $58–63$.

\bibitem{H1}
M.D. Hirschhorn, The Power of $q$. A Personal Journey, Developments in Mathematics, Vol. 49 (Springer, Cham, 2017), xxii+415 pp.

\bibitem{Kathiravan}
T. Kathiravan and S.N. Fathima, On $\ell$-regular bipartitions modulo $\ell$, Ramanujan J., $44$, $(2017)$ $549–558$.

\bibitem{Keith}
W.J. Keith, Congruences for $9$-regular partitions modulo $3$, Ramanujan J., $35$, $(2014)$ $157–164$.

\bibitem{Lin1}
B.L.S. Lin, Arithmetic of the 7-regular bipartition function modulo 3, Ramanujan J., $37$, $(2015)$ $469–478$.

\bibitem{Lin2}
B.L.S. Lin, An infinite family of congruences modulo 3 for 13-regular bipartitions, Ramanujan J., $39$, $(2016)$ $169–178$.

\bibitem{Naika}
M. S. Mahadeva Naika and S. Shivaprasada Nayaka, Some New Congruences for Andrews’ Singular Overpartition Pairs, Vietnam J. Math., $46$ $(2018)$, $609–628$.

\bibitem{Wang}
L. Wang, (2017). Arithmetic properties of $(k,\ell )$ -regular bipartitions, Bull. Aust. Math. Soc., $95$, $(2017)$ $353-364$.

\bibitem{Webb}
J.J. Webb, Arithmetic of the $13$-regular partition function modulo $3$, Ramanujan J., $25$, $(2011)$ $49–56$.

\bibitem{Xia}
E.X.W. Xia, Congruences for some $\ell$-regular partitions modulo $\ell$, J. Number Theory $152$, $(2015)$ $105–117$.

\bibitem{Xia4}
E.X.W. Xia and O.X.M. Yao, Analogues of Ramanujan's partition identities, Ramanujan J., $31$, $(2013)$ $373–396$.

\bibitem{Xia3}
E.X.W. Xia and O.X.M. Yao, Parity results for $9$-regular partitions, Ramanujan J., $34$, $(2014)$ $109–117$.

\bibitem{Xia2}
E.X.W. Xia and O.X.M. Yao, A proof of Keith’s conjecture for $9$-regular partitions modulo $3$, Int. J. Number Theory, $10$, $(2014)$ $669–674$

\bibitem{Xia1}
E.X.W. Xia and O.X.M. Yao, Arithmetic properties for $(s,t)$-regular bipartition functions, J. Number Theory, $171$, $(2017)$ $1–17$.

\bibitem{Yao}
O.X.M. Yao, New congruences modulo powers of $2$ and $3$ for $9$-reguar partitions, J. Number Theory, $142$, $(2014)$ $89–101$.
\end{thebibliography}
\end{document}